\documentclass{amsart}
\usepackage[english,french]{babel}
\usepackage[utf8]{inputenc}
\usepackage[dvips,final]{graphics}
\usepackage{amsmath,amsfonts,amssymb,amsthm,amscd,array,stmaryrd,mathrsfs, mathdots, epigraph}
\usepackage{arydshln}
\usepackage[makeroom]{cancel}
\usepackage{pstricks}
 \usepackage[all]{xy}
 \usepackage{url}
\usepackage{multirow, blkarray}
\usepackage{booktabs}
\usepackage{textcomp}
 \usepackage[final]{epsfig}
 \usepackage{color}
\theoremstyle{plain}
\newtheorem{thm}{Theorem}[section]

\newtheorem{lem}[thm]{Lemma}
\newtheorem{cor}[thm]{Corollary}
\newtheorem{prop}[thm]{Proposition}

\theoremstyle{definition}
\newtheorem*{rem}{Remark}

\begin{document}

\title[Product of two matrices similar to companion matrices]{Product of two matrices similar to companion matrices over sufficiently large fields}
\date{}
\author{Flavien Mabilat}
\email{flavien.mabilat@univ-reims.fr}

\maketitle

\selectlanguage{french}
\begin{abstract}

Dans cette note, on va démontrer qu'une matrice carrée de taille $n$ sur un corps commutatif contenant au moins $2n$ éléments peut s'exprimer comme un produit de deux matrices semblables à des matrices compagnons, c'est-à-dire des matrices ayant le même polynôme minimal et caractéristique, si et seulement si le rang de $A$ est supérieur à $n-2$, en utilisant uniquement des résultats classiques. On donnera également quelques éléments valables sur des corps plus petits.
\\
\end{abstract}

\selectlanguage{english}

\begin{abstract}

In this note, we prove that a square matrix of size $n$ over a field containing at least $2n$ elements can be expressed as the product of two matrices similar to companion matrices, that is to say matrices with the same minimal and characteristic polynomial, if and only if the rank of $A$ is greater than $n-2$, using only classical facts. We will also give some partial results valid over smaller fields.
\\
\end{abstract}
\thispagestyle{empty}


\section{Introduction}

In this note, all fields considered are commutative. Let $\mathbb{K}$ be an arbitrary field. $0_{k,l}$ denotes the zero matrix of $M_{k,l}(\mathbb{K})$ and $I_{n}$ the identity matrix of size $n$. Let $A \in M_{n}(\mathbb{K})$. We denote $\chi_{A}(X):={\rm det}(XI_{n}-A)$ the characteristic polynomial of $A$ (with this definition $\chi_{A}(X)$ is a monic polynomial), $\pi_{A}$ the minimal polynomial of $A$, assumed to be monic, and ${\rm Sp}(A)$ the spectrum of $A$, that is to say the set of the eigenvalues of $A$ over the field $\mathbb{K}$. The rank of $A$ is written ${\rm rg}(A)$. Let $B \in M_{k,l}(\mathbb{K})$, ${}^t B$ is the transpose of $B$. Let $\Omega$ be a finite set, and let $\left|\Omega\right|$ denote its cardinality. Let $n \geq 1$ and $P(X):=X^{n}+\sum_{i=0}^{n-1} a_{i}X^{i} \in \mathbb{K}[X]$. The companion matrix of $P$ is the matrix $C(P)$ defined as follows:
$C(P):=\begin{pmatrix}
   0 & 0 & \ldots & 0 & -a_{0}  \\
      1 & 0 & \ldots & 0 & -a_{1}   \\
		  0 & \ddots & \ddots & \vdots & \vdots  \\
		  \vdots & \ddots & \ddots & 0 & -a_{n-2}  \\
		  0 & \ldots & 0 & 1 &  -a_{n-1} \\
   \end{pmatrix}$.
\\
\\
\\\indent These matrices play an important role in the study of square matrices since they are the master pieces of the well-known Frobenius normal form. They are also a useful tool to study polynomials since $P$ is the characteristic polynomial of $C(P)$. Hence, they are matrices of particular interest. This naturally leads us to consider products of companion matrices. We know several results in this direction. For instance, we have an explicit combinatorial formula for the coefficients of product of companion matrices (see \cite{LD}) and several results concerning the eigenvalues of such products (see \cite{KV}). Note that we also have several results about the decomposition of companion matrices (see \cite{F}). 
\\
\\\indent Considering all these elements, a question arises naturally: Can we decompose a square matrix as a product of two companion matrices ? In fact, the particular form of these matrices is a little too restrictive to have a sufficiently large class of matrices satisfying the desired property. For instance, we have the following equality:
\[\begin{pmatrix}
   0 & 0 & a \\
    1 & 0 & b  \\
		0 & 1 & c \\
   \end{pmatrix}\begin{pmatrix}
   0 & 0 & e \\
    1 & 0 & f  \\
		0 & 1 & g \\
   \end{pmatrix}=\begin{pmatrix}
   0 & a & ag \\
    0 & b & e+bg  \\
		1 & c & f+cg \\
   \end{pmatrix}.\]
	
However, we can consider a related but less restrictive problem: Can we decompose a square matrix as a product of two matrices similar to companion matrices, that is to say matrices whose minimal and characteristic polynomials are identical ? Surprisingly, it seems that there is no answer to this natural question. In this note, we will study this interrogation and prove the following result:

\begin{thm}
\label{11}

Let $n$ be a positive integer and $\mathbb{K}$ be a commutative field verifying $\left|\mathbb{K}\right| \geq 2n$. A matrix $A \in M_{n}(\mathbb{K})$ can be written as a product of two matrices similar to companion matrices if and only if ${\rm rg}(A) \geq n-2$.

\end{thm}

This theorem is proved in the next section, using only classical arguments. Besides, we give, in Section \ref{finite}, some elements concerning the case of smaller fields.

\section{Proof of the main result}

The aim of this section is to prove Theorem \ref{11} and to give some elements concerning the decomposition studied in this result.

\subsection{Matrices similar to companion matrices} 

We give here some properties verified by the matrices similar to companion matrices. We begin with a useful characterization of them.

\begin{prop}
\label{21}

Let $n$ be a positive integer and $A \in M_{n}(\mathbb{K})$. The matrix $A$ is similar to a companion matrix if and only if $\pi_{A}=\chi_{A}$.

\end{prop}

\begin{proof}

Suppose that there exist $Q \in GL_{n}(\mathbb{K})$ and $P \in \mathbb{K}[X]$ such that $A=QC(P)Q^{-1}$. We have $\pi_{A}=\pi_{C(P)}$ and $\chi_{A}=\chi_{C(P)}$. Besides, since $C(P)$ is a companion matrix, $\pi_{C(P)}=\chi_{C(P)}$ (see for instance \cite{HJ} Theorem 3.3.14). Hence, $\pi_{A}=\chi_{A}$.
\\
\\Suppose that $\pi_{A}=\chi_{A}$. Let $u$ be the endomorphism of $\mathbb{K}^{n}$ whose matrix in the canonical basis is $A$. There exists $x \in \mathbb{K}^{n}$ such that $\pi_{u,x}=\pi_{u}=\pi_{A}$, with $\pi_{u,x}$ the monic polynomial which generates the ideal $\{P \in \mathbb{K}[X],~P(u)(x)=0\}$. Hence, the family $(x,u(x),\ldots,u^{n-1}(x))$ is linearly inependant, since $n-1< {\rm deg}(\pi_{u,x})=n$. So, $\mathcal{B}:=(x,u(x),\ldots,u^{n-1}(x))$ is a basis of $\mathbb{K}^{n}$ and the matrix of $u$ in $\mathcal{B}$ is a companion matrix.

\end{proof}

Note that $\pi_{A}=\chi_{A}$ if and only if ${\rm deg}(\pi_{A})=n$ (this is an easy consequence of Cayley-Hamilton theorem). We also have the following characterization:

\begin{prop}[\cite{FL} Lemma 3.1]
\label{22}

Let $n$ be a positive integer and $A \in M_{n}(\mathbb{K})$. The polynomials $\pi_{A}$ and $\chi_{A}$ coincide if and only if, for every $(x_{1},\ldots,x_{n}) \in \mathbb{K}^{n}$, there exists $(P,Q) \in M_{n,1}(\mathbb{K})^{2}$ such that, for all $1 \leq k \leq n$, $x_{k}={}^t\! Q A^{k}P$.

\end{prop}

The matrices similar to companion matrices appear in particular in the following well-known matrix-completion result:

\begin{thm}[Farahat-Ledermann, \cite{FL} Theorem 3.4]
\label{23}

Let $n$ be a positive integer, $B \in M_{n-1}(\mathbb{K})$ such that $\pi_{B}=\chi_{B}$, and $R(X)$ a given monic polynomial of degree $n$ over $\mathbb{K}$. Then there exists an $n \times n$ matrix having $B$ in the top left-hand corner, whose characteristic polynomial is $R(X)$.

\end{thm}

\noindent We conclude this subsection with the easy result below.

\begin{prop}
\label{24}

Let $n$ be a positive integer. Let $A \in M_{n}(\mathbb{K})$ similar to a companion matrix. The rank of $A$ is equal to $n$ or $n-1$.

\end{prop}

\begin{proof}

There exists $P \in \mathbb{K}[X]$ such that $A$ is similar to $C(P)$. We can extract from $C(P)$ a submatrix equal to the identity matrix of size $n-1$. Hence, ${\rm rg}(A)={\rm rg}(C(P)) \geq n-1$.

\end{proof}

\subsection{Proof of the decomposition theorem} 

The aim of this subsection is to prove Theorem \ref{11}. We decompose it into several subresults presented as lemmas. 

\subsubsection{Proof of the necessary condition}
	
We consider here the case of the matrices whose rank is lower than $n-3$ when $n$ is a positive integer greater than three.

\begin{lem}
\label{26}

Let $n \geq 3$ and $A \in M_{n}(\mathbb{K})$. If ${\rm rg}(A) \leq n-3$ then $A$ cannot be written as a product of two matrices similar to companion matrices.

\end{lem}

\begin{proof}

Suppose that there exists $(B,C) \in M_{n}(\mathbb{K})^{2}$ such that $A=BC$, $\pi_{B}=\chi_{B}$ and $\pi_{C}=\chi_{C}$. The matrix $B$ is not invertible, since otherwise we would have, by Proposition \ref{24}, $n-1 \leq {\rm rg}(C)={\rm rg}(A) \leq n-3$. Since the polynomial $\pi_{B}$ is equal to $\chi_{B}$, there exist $P \in GL_{n}(\mathbb{K})$ and $(b_{1},\ldots,b_{n}) \in \mathbb{K}^{n}$ such that $PBP^{-1}=\begin{pmatrix}
   0 & 0 & \ldots & 0 & b_{1}  \\
      1 & 0 & \ldots & 0 & b_{2}   \\
		  0 & \ddots & \ddots & \vdots & \vdots  \\
		  \vdots & \ddots & \ddots & 0 & b_{n-1}  \\
		  0 & \ldots & 0 & 1 &  b_{n} \\
   \end{pmatrix}$ (by Proposition \ref{21}).
\\
\\We have $A':=PAP^{-1}=(PBP^{-1})(PCP^{-1})$. Let $A':=(a_{i,j})_{1 \leq i,j \leq n}$ and $PCP^{-1}:=(c_{i,j})_{1 \leq i,j \leq n}$. Since $B$ is not invertible, $b_{1}=0$. We have:
\begin{eqnarray*}
A' &=& \begin{pmatrix}
   0 & 0 & \ldots & 0 & 0  \\
      1 & 0 & \ldots & 0 & b_{2}   \\
		  0 & \ddots & \ddots & \vdots & \vdots  \\
		  \vdots & \ddots & \ddots & 0 & b_{n-1}  \\
		  0 & \ldots & 0 & 1 &  b_{n} \\
   \end{pmatrix}\begin{pmatrix}
   c_{1,1} & \ldots & c_{1,n}  \\
      \vdots &   & \vdots   \\
		  c_{n,1} & \ldots & c_{n,n} \\
   \end{pmatrix} \\
	     &=& \begin{pmatrix}
   0  & \ldots  & 0  \\
		  c_{1,1}+b_{2}c_{n,1} & \ldots & c_{1,n}+b_{2}c_{n,n} \\
			\vdots &   & \vdots  \\
			c_{n-1,1}+b_{n}c_{n,1} & \ldots & c_{n-1,n}+b_{n}c_{n,n} \\
   \end{pmatrix}.
\end{eqnarray*}

\noindent For all $1 \leq i \leq n-1$ and for all $1 \leq j \leq n$, we have $c_{i,j}+b_{i+1}c_{n,j}=a_{i+1,j}$ and so $c_{i,j}=a_{i+1,j}-b_{i+1}c_{n,j}$. Besides, $a_{1,j}=0$ for all $j$. Hence, setting $L:=\begin{pmatrix}
   c_{n,1} & \ldots & c_{n,n}  \\
   \end{pmatrix}$, we obtain:
\[PCP^{-1}=\begin{pmatrix}
   a_{2,1}-b_{2}c_{n,1} & \ldots & a_{2,n}-b_{2}c_{n,n}  \\
      \vdots & & \vdots \\
		  a_{n,1}-b_{n}c_{n,1} & \ldots & a_{n,n}-b_{n}c_{n,n}  \\
			c_{n,1} & \ldots & c_{n,n}  \\
   \end{pmatrix}=\underbrace{\begin{pmatrix}
   a_{2,1} & \ldots & a_{2,n}  \\
      \vdots & & \vdots \\
		  a_{n,1} & \ldots & a_{n,n}  \\
			a_{1,1} & \ldots & a_{1,n}  \\
   \end{pmatrix}}_{D}+\underbrace{\begin{pmatrix}
-b_{2}~L \\
	\vdots  \\
	-b_{n}~L \\
	  L
   \end{pmatrix}}_{E}.\]
\noindent We have ${\rm rg}(D)={\rm rg}(A) \leq n-3$ and ${\rm rg}(E)=1$. Hence, ${\rm rg}(C) \leq {\rm rg}(D)+{\rm rg}(E) \leq n-2$. However, by Proposition \ref{24}, ${\rm rg}(C) \geq n-1$. So, we have a contradiction. 
\\
\\Hence, $A$ cannot be written as a product of two matrices similar to companion matrices.

\end{proof}

\subsubsection{The case of invertible matrices}

We begin by the case $n \leq 2$ for invertible and non-invertible matrices.

\begin{lem}
\label{25}

If $n=1$ or $n=2$ then every square matrix $A \in M_{n}(\mathbb{K})$ can be written as a product of two matrices similar to companion matrices.

\end{lem}

\begin{proof}

Suppose that $n=1$. There exists $a \in \mathbb{K}$ such that $A=(a)$. We have:
\[A=(a) \times (1)=C(X-a) \times C(X-1).\]

\noindent Suppose that $n=2$. There exists $(a,b,c,d) \in \mathbb{K}^{4}$ such that $A=\begin{pmatrix}
   a & b \\
    c & d   \\
   \end{pmatrix}$. First, we suppose $a \neq 0$. We have: 
	\[A=\begin{pmatrix}
   0 & a \\
   1 & c  \\
   \end{pmatrix}\begin{pmatrix}
   0 & d-a^{-1}cb \\
   1 & 	a^{-1}b \\
   \end{pmatrix}=C(X^{2}-cX-a) \times C(X^{2}-a^{-1}bX-d+a^{-1}cb).\]

\noindent Now, we suppose $a=0$. We have: $A=\begin{pmatrix}
   0 & b \\
   1 & d-1  \\
   \end{pmatrix}\begin{pmatrix}
   c &  1 \\
   0 & 	1 \\
   \end{pmatrix}$. We set $B:=\begin{pmatrix}
   c &  1 \\
   0 & 	1 \\
   \end{pmatrix}$. An immediate calculation gives $\chi_{B}(X)=(X-c)(X-1)$. By Cayley-Hamilton theorem, $\pi_{B}$ divides $\chi_{B}$ and $(X-1)(B)$ and $(X-c)(B)$ are different from $0_{2,2}$. Hence, $\pi_{B}=\chi_{B}$ and $A=C(X^{2}-(d-1)X-b) \times B$. By Proposition \ref{21}, we have the desired result.

\end{proof}

\begin{rem}
{\rm \noindent If $a=b=0$, we have: $A=\begin{pmatrix}
   0 & 0 \\
   1 & c  \\
   \end{pmatrix}\begin{pmatrix}
   0 &  d \\
   1 & 	0 \\
   \end{pmatrix}=C(X^{2}-cX) \times C(X^{2}-d)$.
	}
	\end{rem}

\noindent We now consider the case of invertible matrices.

\begin{lem}
\label{27}

Let $n \geq 1$, $\mathbb{K}$ a commutative field with $\left|\mathbb{K}\right| \geq 2n$ and $A \in GL_{n}(\mathbb{K})$. The matrix $A$ can be written as a product of two (invertible) matrices similar to companion matrices.

\end{lem}

\begin{proof}

First, we consider the case of invertible scalar matrices. Let $A \in M_{n}(\mathbb{K})$ be an invertible scalar matrix. There exists $\lambda \in \mathbb{K}$, $\lambda \neq 0$, such that $A=\lambda I_{n}$. We choose $n$ pairwise distinct elements $\mu_{1},\ldots,\mu_{n}$ in $\mathbb{K}-\{0\}$ (this choice is possible since $\left|\mathbb{K}\right| \geq 2n$). We define the two following matrices:
\[B:=\begin{pmatrix}
   \lambda \mu_{1} &  &   \\
       &  \ddots &    \\
		  &  & \lambda \mu_{n}  \\
   \end{pmatrix}~~{\rm and}~~C:=\begin{pmatrix}
   \mu_{1}^{-1} &  &   \\
       &  \ddots &    \\
		  &  & \mu_{n}^{-1} \\
   \end{pmatrix}.\]
\noindent We have the following equalities $A=BC$, $\pi_{B}(X)=\chi_{B}(X)=\prod_{i=1}^{n} (X-\lambda \mu_{i})$ (since $\lambda \mu_{i} \neq \lambda \mu_{j}$ for $i \neq j$) and $\pi_{C}(X)=\chi_{C}(X)=\prod_{i=1}^{n} (X-\mu_{i}^{-1})$ (since $\mu_{i}^{-1} \neq \mu_{j}^{-1}$ for $i \neq j$).
\\
\\We return to the general case. We proceed by induction on $n$. If $n=1$, the result is true. Suppose that there exists an integer $n \geq 1$ such that, for all commutative field $\mathbb{K}$ satisfying $\left|\mathbb{K}\right| \geq 2n$, every invertible matrix of size $n$ over $\mathbb{K}$ can be written as a product of two invertible matrices similar to companion matrices. Let $\mathbb{K}$ be a commutative field $\mathbb{K}$ satisfying $\left|\mathbb{K}\right| \geq 2(n+1)=2n+2$. Let $A \in GL_{n+1}(\mathbb{K})$. 
\\
\\If $A$ is a scalar matrix then $A$ can be written as a product of two invertible matrices similar to companion matrices. We now assume that $A$ is non-scalar (this case is possible since $n+1 \geq 2$). Let $u$ be the endomorphism of $\mathbb{K}^{n+1}$ whose matrix in the canonical basis is $A$. Since $A$ is non-scalar, there exists $v \in \mathbb{K}^{n+1}$ such that $(v,u(v))$ is linealy independant. Thus, $(v,u(v)-v)$ is linearly independant. We complete this family into a basis $\mathcal{C}:=(v,u(v)-v,w_{3},\ldots,w_{n+1})$ of $\mathbb{K}^{n+1}$. The matrix of $u$ in $\mathcal{C}$ is similar to $A$ and is equal to:
\[B:=\left (
   \begin{array}{c|c}
      1 & x \\
      \hline
      y & C\\ 
   \end{array}
\right),\]
\noindent with $C:=(c_{i,j}) \in M_{n}(\mathbb{K})$, $x \in M_{1,n}(\mathbb{K})$ and $y:={}^t\!\begin{pmatrix}
   1 & 0 & \ldots & 0 \\
   \end{pmatrix} \in M_{n,1}(\mathbb{K})$.
\\
\\We set: $x:=\begin{pmatrix}
   x_{1} & \ldots & x_{n}  \\
   \end{pmatrix}$ and $T:=\left(\begin{array}{c|c}
      \begin{array}{cc}
      1 & 0 \\
      -1 & 1 \\ 
   \end{array} & 0_{2,n-1} \\
      \hline
      0_{n-1,2} & I_{n-1}\\ 
   \end{array}
\right)$. An immediate calculation gives: 
\[TB=\begin{pmatrix}
   1 & x_{1} & \ldots & x_{n}  \\
	 0 & c_{1,1}-x_{1} & \ldots & c_{1,n}-x_{n} \\
	 0 & c_{2,1} & \ldots & c_{2,n} \\
		\vdots & \vdots &   & \vdots \\
	 0 & c_{n,1} & \ldots & c_{n,n} \\
   \end{pmatrix}=\left (
   \begin{array}{c|c}
      1 & x \\
      \hline
      0_{n,1} & C-yx\\ 
   \end{array}
\right).\]
\noindent We have ${\rm det}(B)={\rm det}(TB)={\rm det}(C-yx)$. This implies ${\rm det}(C-yx) \neq 0$. By induction hypothesis, there exists $(D,E) \in GL_{n}(\mathbb{K})^{2}$ such that $C-yx=DE$, $\pi_{D}=\chi_{D}$ and $\pi_{E}=\chi_{E}$. 
\\
\\Besides, there exists $\mu \in \mathbb{K}$ different from 0 such that $\mu \notin {\rm Sp}(D)$ and $\mu^{-1} \notin {\rm Sp}(E)$. Indeed, the set ${\rm Sp}(D) \cup \{\lambda,~\lambda^{-1} \in {\rm Sp}(E)\}$ contains at most $2n$ elements. However, $\left|\mathbb{K}-\{0\}\right| \geq 2n+1$. Hence, there exists $\mu \in \mathbb{K}-\{0\}$ such that $\mu \notin {\rm Sp}(D) \cup \{\lambda,~\lambda^{-1} \in {\rm Sp}(E)\}$. This element verify the desired property.
\\
\\We have the following equalities:
\[\underbrace{\left(
   \begin{array}{c|c}
      \mu & 0_{1,n} \\
      \hline
       \mu y & D \\ 
   \end{array}\right)}_{F}\underbrace{\left(
   \begin{array}{c|c}
      \mu^{-1} & \mu^{-1}x \\
      \hline
       0_{n,1} & E\\ 
   \end{array}\right)}_{G}=\left(\begin{array}{c|c}
      1 & x \\
      \hline
      y & yx+DE\\ 
   \end{array}
\right)=\left (
   \begin{array}{c|c}
      1 & x \\
      \hline
      y & C\\ 
   \end{array}
\right)=B.\]

\noindent We have $\chi_{F}(X)=(X-\mu)\chi_{D}(X)$. By Cayley-Hamilton theorem, the polynomial $\pi_{F}(X)$ divides $\chi_{F}(X)$. Besides, $\pi_{F}(X)$ is a multiple of the lowest common multiple of $(X-\mu)$ and $\pi_{D}(X)$ (since $F$ is block-triangular). Since $\mu \notin {\rm Sp}(D)$, this lowest common multiple is $(X-\mu)\pi_{D}(X)=\chi_{F}(X)$. Hence, $\pi_{F}$ is equal to $\chi_{F}$. A similar argument gives $\pi_{G}=\chi_{G}$. Hence, $A$ can be written as a product of two matrices similar to companion matrices.
\\
\\By induction, the result is proved.

\end{proof}

\subsubsection{Two simple cases of non-invertible matrices}

The following two results do not play a role in the proof of the main theorem, but they provide two simple applications of the preceding lemma.

\begin{cor}
\label{28}

Let $n \geq 1$, $\mathbb{K}$ a commutative field with $\left|\mathbb{K}\right| \geq 2n-2$ and $A \in M_{n}(\mathbb{K})$ verifying ${\rm rg}(A)=n-1$ and $\chi_{A}(X)=XR(X)$, with $R(0) \neq 0$. The matrix $A$ can be written as a product of two matrices similar to companion matrices.

\end{cor}

\begin{proof}

If $n=1$ or $n=2$ then the result follows from Lemma \ref{25}. We suppose that $n \geq 3$. Let $u$ be the endomorphism of $\mathbb{K}^{n}$ whose matrix in the canonical basis is $A$. Since ${\rm rg}(A)=n-1$, 0 is an eigenvalue of $u$. Hence, there exists $v \in \mathbb{K}^{n}-\{0\}$ such that $u(v)=0$. We complete the family $(v)$ into a basis $\mathcal{C}:=(v,w_{2},\ldots,w_{n})$ of $\mathbb{K}^{n}$. The matrix of $u$ in $\mathcal{C}$ is similar to $A$ and is equal to:
\[B:=\left (
   \begin{array}{c|c}
      0 & x \\
      \hline
      0_{n-1,1} & C\\ 
   \end{array}
\right),\] 
\noindent with $C:=(c_{i,j}) \in M_{n-1}(\mathbb{K})$ and $x \in M_{1,n-1}(\mathbb{K})$. Since $XR(X)=\chi_{A}(X)=X\chi_{C}(X)$, $C$ is invertible. Moreover, $\left|\mathbb{K}\right| \geq 2n-2=2(n-1)$. By Lemma \ref{27}, there exists $(D,E) \in GL_{n-1}(\mathbb{K})^{2}$ such that $C=DE$, $\pi_{D}=\chi_{D}$ and $\pi_{E}=\chi_{E}$. Besides, there exists $\mu \in \mathbb{K}-\{0\}$ such that $\mu \notin {\rm Sp}(D)$ ($\mu$ exists since $n \geq 3$ and $\left|\mathbb{K}- \{0\}\right| \geq 2n-3>n-1 \geq \left|{\rm Sp}(D)\right|$). We have: 
\[\underbrace{\left (
   \begin{array}{c|c}
      \mu & 0_{1,n-1} \\
      \hline
      0_{n-1,1} & D\\ 
   \end{array}
\right)}_{F}\underbrace{\left (
   \begin{array}{c|c}
      0 & \mu^{-1}x \\
      \hline
      0_{n-1,1} & E\\ 
   \end{array}
\right)}_{G}=\left (
   \begin{array}{c|c}
      0 & x \\
      \hline
      0_{n-1,1} & DE\\ 
   \end{array}
\right)=B.\]
\noindent The matrix $F$ is block-diagonal and $\mu \notin {\rm Sp}(D)$. Hence, $\pi_{F}=\chi_{F}$. The matrix $E$ is invertible. Hence, $0 \notin {\rm Sp}(E)$ and $\pi_{G}=\chi_{G}$ (the argument is the same as the one used at the end of the proof of the previous lemma).

\end{proof}

\begin{cor}
\label{29}

Let $n \geq 2$, $\mathbb{K}$ a commutative field with $\left|\mathbb{K}\right| \geq 2n-3$ and $A \in M_{n}(\mathbb{K})$ verifying ${\rm rg}(A)=n-2$ and $\chi_{A}(X)=X^{2}R(X)$, with $R(0) \neq 0$. The matrix $A$ can be written as a product of two matrices similar to companion matrices.

\end{cor}

\begin{proof}

If $n=2$ then $A=0_{2,2}$ and the result follows from Lemma \ref{25}. We suppose that $n \geq 3$. Let $u$ be the endomorphism of $\mathbb{K}^{n}$ whose matrix in the canonical basis is $A$. Since ${\rm rg}(A)=n-2$, there exist $t,v \in \mathbb{K}^{n}$ such that $(t,v)$ is linearly independant and $u(t)=u(v)=0$. We complete the family $(t,v)$ into a basis $\mathcal{C}:=(t,v,w_{3},\ldots,w_{n})$ of $\mathbb{K}^{n}$. The matrix of $u$ in $\mathcal{C}$ is similar to $A$ and is equal to:
\[B:=\left (
   \begin{array}{c|c}
      0_{2,2} & L \\
      \hline
      0_{n-2,2} & C\\ 
   \end{array}
\right),\] 
\noindent with $C:=(c_{i,j}) \in M_{n-2}(\mathbb{K})$ and $L={}^t\!\begin{pmatrix}
   L_{1} & L_{2}
   \end{pmatrix} \in M_{2,n-2}(\mathbb{K})$. Since $X^{2}R(X)=\chi_{A}(X)=X^{2}\chi_{C}(X)$, $C$ is invertible. Besides, $\left|\mathbb{K}\right| \geq 2n-3>2(n-2)$. By Lemma \ref{27}, there exists $(D,E) \in GL_{n-2}(\mathbb{K})^{2}$ such that $C=DE$, $\pi_{D}=\chi_{D}$ and $\pi_{E}=\chi_{E}$. Besides, there exists $(\lambda,\mu) \in (\mathbb{K}-\{0\})^{2}$ such that $\lambda \notin {\rm Sp}(D)$ and $\mu \notin {\rm Sp}(E)$ (since $n \geq 3$ and $\left|\mathbb{K}- \{0\}\right| \geq 2n-4>n-2 \geq \left|{\rm Sp}(D)\right|=\left|{\rm Sp}(E)\right|$). We have: 
\[\underbrace{\left (
   \begin{array}{c|c|c}
      \lambda & 0 & 0_{1,n-2} \\
      \hline
      0 & 0 & L_{2}E^{-1}\\
			\hline
			 0_{n-2,1} & 0_{n-2,1} & D\\
   \end{array}
\right)}_{F}\underbrace{\left (
   \begin{array}{c|c|c}
      0 & 0 & \lambda^{-1}L_{1} \\
      \hline
      0 & \mu & 0_{1,n-2} \\
			\hline
			 0_{n-2,1} & 0_{n-2,1} & E\\
   \end{array}
\right)}_{G}=\left (
   \begin{array}{c|c|c}
      0 & 0 & L_{1} \\
      \hline
      0 & 0 & L_{2} \\
			\hline
			 0_{n-2,1} & 0_{n-2,1} & DE\\
   \end{array}
\right)=B.\]
\noindent The matrices $F$ and $G$ are block-triangular, $D$ and $E$ are invertible, $\lambda \notin {\rm Sp}(D)$ and $\mu \notin {\rm Sp}(E)$. Hence, $\pi_{F}=\chi_{F}$ and $\pi_{G}=\chi_{G}$ (the argument is the same as the one used at the end of the proof of Lemma \ref{27}).

\end{proof}

\subsubsection{The general case for non-invertible matrices}

We need the following classical result:

\begin{thm}[Sourour-Tang, \cite{SK} Theorem 1] 

Let $A$ be an  $n \times n$ singular matrix over an arbitrary commutative field $\mathbb{K}$ and let $\beta_{j}$ and $\gamma_{j}$ ($1 \leq j \leq n$) be elements of $\mathbb{K}$. If $A$ is not a nonzero $2\times 2$ nilpotent matrix, then $A$ can be factored as a product $BC$ where the eigenvalues of $B$ and $C$ are $\beta_{1},\ldots,\beta_{n}$ and $\gamma_{1},\ldots,\gamma_{n}$ respectively, if and only if the number of zeros $m$ among $\beta_{1},\ldots,\beta_{n},\gamma_{1},\ldots,\gamma_{n}$ is not less than the dimension of the null space of $A$. If $A$ is a nonzero $2 \times 2$ nilpotent matrix then $A$ can be factored as above if and only if $1 \leq m \leq 3$.

\end{thm}

In fact, we are far from needing the full strength of this theorem. However, it is more convenient to use it as is than to reproduce substantial parts of its proof. It allows us to prove the final lemma that we need.

\begin{lem}

Let $n$ be a positive integer and $\mathbb{K}$ be a commutative field verifying $\left|\mathbb{K}\right| \geq n$. Let $A$ be a square matrix of size $n$ whose rank is equal to $n-1$ or $n-2$. The matrix $A$ can be written as a product of two matrices similar to companion matrices. 

\end{lem}

\begin{proof}

If $n=1$ or $n=2$ then the result follows from Lemma \ref{25}. We suppose now $n \geq 3$. Let $\lambda_{1},\lambda_{2},\ldots\lambda_{n}$ be $n$ pairwise distinct elements of $\mathbb{K}$, with $\lambda_{1}=0$. By Sourour-Tang theorem, there exist $B,C$ such that $A=BC$ and such that ${\rm Sp}(B)={\rm Sp}(C)=\{0,\lambda_{2},\ldots\lambda_{n}\}$. We have $\chi_{B}(X)=\chi_{C}(X)=\prod_{i=1}^{n} (X-\lambda_{i})$ and $\pi_{B}(X)=\pi_{C}(X)=\prod_{i=1}^{n} (X-\lambda_{i})$ since $\lambda_{i} \neq \lambda_{j}$ for $i \neq j$.

\end{proof}

\noindent All the lemmas contained in this section give us Theorem \ref{11}.

\subsection{Some concluding remarks}

The main focus of this note is the study of products of two matrices that are similar to companion matrices. However, in light of the preceding results, one can readily extend the discussion to products of $k$ matrices of this form.

\begin{prop}

Let $n \geq 2$ and $2 \leq k \leq n$. Let $\mathbb{K}$ be a commutative field verifying $\left|\mathbb{K}\right| \geq n$. Let $A$ be a square matrix of size $n$. If ${\rm rg}(A)=n-k$ then $A$ can be written as a product of $k$ matrices similar to companion matrices. 

\end{prop}

\begin{proof}

If $n=2$ then the result follows from Lemma \ref{25}. We suppose now $n \geq 3$. We proceed by finite induction on $k$. 
\\
\\Suppose that $k=2$. The result follows from the previous lemma. 
\\
\\Suppose that there exists $2 \leq k \leq n-1$ such that every square matrix $B$ of size $n$ whose rank is equal to $n-k$ can be written as a product of $k$ matrices similar to companion matrices. Let $A$ be a square matrix of size $n$ satisfying ${\rm rg}(A)=n-k-1=n-(k+1)$. 
\\
\\Let $\lambda_{1},\lambda_{2},\ldots\lambda_{n}$ be $n$ pairwise distinct elements of $\mathbb{K}$, with $\lambda_{1}=0$. By Sourour-Tang theorem, there exist $B,C$ such that $A=BC$ and such that ${\rm Sp}(B)=\{0,\lambda_{k+1},\ldots,\lambda_{n}\}$, with $0$ an eigenvalue of multiplicity $k$, and ${\rm Sp}(C)=\{0,\lambda_{2},\ldots\lambda_{n}\}$. We have ${\rm rg}(B)=n-k$. Hence, by induction hypothesis, $B$ can be written as a product of $k$ matrices similar to companion matrices. Besides, $\pi_{C}(X)=\chi_{C}(X)=\prod_{i=1}^{n} (X-\lambda_{i})$, since $\lambda_{i} \neq \lambda_{j}$ for $i \neq j$. Hence, by Proposition \ref{21}, $C$ is similar to a companion matrix. Consequently, $A$ can be written as a product of $k+1$ matrices similar to companion matrices. 
\\
\\By induction, the result is proved.

\end{proof}

\begin{rem}
{\rm By Theorem \ref{11}, a square matrix of size $n$, over a field $\mathbb{K}$ satisfying $\left|\mathbb{K}\right| \geq 2n$, whose rank is equal to $n$ or $n-1$ is the product of two matrices similar to companion matrices. Besides, by Proposition \ref{24}, the rank of a companion matrix is equal to $n$ or $n-1$. Hence, by immediate induction, a square matrix of size $n$ whose rank is equal to $n$ or $n-1$ can be written, for every $k \geq 2$, as a product of $k$ matrices similar to companion matrices. If ${\rm rg}(A)=n-h$, with $2 \leq h \leq n$, then, by the previous proposition and by immediate induction, $A$ can be expressed, for every $l \geq h$, as a product of $l$ matrices similar to companion matrices. 
}
\end{rem}

\noindent We conclude this section by the two following remarks:
\begin{itemize}
\item In general, the decomposition $A=BC$ with $B,C$ similar to companion matrices is not unique. For example: \[\begin{pmatrix}
    0 & 0 \\[4pt]
    1 & 1 
   \end{pmatrix}\begin{pmatrix}
    0 & 1 \\[4pt]
    1 & 0 
   \end{pmatrix}=\begin{pmatrix}
    0 & 0 \\[4pt]
    1 & 1 
   \end{pmatrix}=\begin{pmatrix}
    0 & 0 \\[4pt]
    1 & 0 
   \end{pmatrix}\begin{pmatrix}
    1 & 1 \\[4pt]
    0 & 1 
   \end{pmatrix}.\]
	\\
\item If $A=BC$ with $B,C$ similar to companion matrices then $B$ and $C$ don't commute in general. For instance:
\[\begin{pmatrix}
    0 & 0 \\[4pt]
    1 & 1 
   \end{pmatrix}\begin{pmatrix}
    0 & 1 \\[4pt]
    1 & 0 
   \end{pmatrix}=\begin{pmatrix}
    0 & 0 \\[4pt]
    1 & 1 
   \end{pmatrix};~~~~~~\begin{pmatrix}
    0 & 1 \\[4pt]
    1 & 0 
   \end{pmatrix}\begin{pmatrix}
    0 & 0 \\[4pt]
    1 & 1 
   \end{pmatrix}=\begin{pmatrix}
    1 & 1 \\[4pt]
    0 & 0 
   \end{pmatrix}.\]
\end{itemize}

\noindent Note that these two remarks don't use the cardinality of the field.

\section{Improvement of the field cardinality bounds in special cases}
\label{finite}

The aim of this section is to give some partial decomposition results over finite fields smaller than those used in Theorem \ref{11}.
\\
\\\indent First, note that lemmas \ref{25} and \ref{26} are true for all commutative fields. Besides, we can easily decompose invertible diagonalizable matrices (and so extend the result used in the beginning of the proof of Theorem \ref{11}).

\begin{prop}

Let $A$ be an invertible and diagonalizable matrix over an arbitrary commutative field $\mathbb{K}$. The matrix $A$ can be written as a product of two matrices similar to companion matrices. 

\end{prop}

\begin{proof}

There exist a positive integer $n$ and $(\lambda_{1},\lambda_{2},\ldots,\lambda_{n}) \in (\mathbb{K}-\{0\})^{n}$ such that $A$ is similar to the matrix $B:=\begin{pmatrix}
    \lambda_{1} &  & \\[4pt]
     & \ddots & \\
		& & \lambda_{n}
   \end{pmatrix}$. We have $B=\underbrace{\begin{pmatrix}
   0 & 0 & \ldots & 0 & 1  \\
      1 & 0 & \ldots & 0 & 0  \\
		  0 & \ddots & \ddots & \vdots & \vdots  \\
		  \vdots & \ddots & \ddots & 0 & 0  \\
		  0 & \ldots & 0 & 1 &  0 \\
   \end{pmatrix}}_{D}\underbrace{\begin{pmatrix}
   0 & \lambda_{2} & 0 & \ldots &  0  \\
	   \vdots & \ddots & \ddots & \ddots &  \vdots  \\
      \vdots &  & \ddots & \ddots & 0  \\
		  0 & \ldots & \ldots & 0 & \lambda_{n}  \\
		  \lambda_{1} & 0 & \ldots & 0 & 0  \\
   \end{pmatrix}}_{E}$. The matrix $D$ is equal to $C(X^{n}-1)$ and the matrix $E$ is similar to a companion matrix since we have: $HEH^{-1}={}^t\!C(X^{n}-\lambda_{1}\ldots \lambda_{n})$, with $H:=\begin{pmatrix}
   0 & \frac{1}{\lambda_{1}} & 0 & \ldots &  0  \\
	   \vdots & \ddots & \frac{\lambda_{2}}{\lambda_{1}} & \ddots &  \vdots  \\
      \vdots &  & \ddots & \ddots & 0  \\
		  0 & \ldots & \ldots & 0 & \frac{\lambda_{2}\ldots \lambda_{n-1}}{\lambda_{1}}  \\
		  \frac{\lambda_{2}\ldots \lambda_{n}}{\lambda_{1}} & 0 & \ldots & 0 & 0  \\
   \end{pmatrix}$.

\end{proof}

\begin{prop}[\cite{M} Lemma 2.1] 

Let $A:=\begin{pmatrix}
   a_{1,1} & a_{1,2} & a_{1,3} & \ldots & a_{1,n}  \\
      0 & a_{2,2} & a_{2,3} & \ldots & a_{2,n}   \\
		  \vdots & \ddots & \ddots &  & \vdots  \\
		  \vdots &  & \ddots & a_{n-1,n-1} & a_{n-1,n}  \\
		  0 & \ldots & \ldots & 0 &  0 \\
   \end{pmatrix} \in M_{n}(\mathbb{K})$ with $a_{i,i} \neq 0$ and $P \in \mathbb{K}[X]$ a monic polynomial of degree $n$. It exists $(\alpha_{1},\ldots,\alpha_{n}) \in \mathbb{K}^{n}$ such that the characteristic polynomial of the matrix $B=\begin{pmatrix}
   \alpha_{1} & \alpha_{2} & \ldots & \ldots & \alpha_{n}  \\
       a_{1,1} & a_{1,2} & \ldots & \ldots & a_{1,n}  \\
		   0 & a_{2,2} & \ldots & \ldots & a_{2,n}  \\
			 \vdots & \ddots & \ddots &  &  \vdots \\
		  0 & \ldots & 0 & a_{n-1,n-1} & a_{n-1,n}  \\
		   \end{pmatrix}$ is equal to $P$.

\end{prop}

\begin{cor}

Let $n \geq 2$ and $\mathbb{K}$ a finite field. Let $A$ be a triangularizable square matrix of size $n$ satisfying $\chi_{A}(X)=XR(X)$, with $R(0) \neq 0$. The matrix $A$ can be written as a product of two matrices similar to companion matrices. 

\end{cor}

\begin{proof}

The matrix $A$ is similar to $B:=\begin{pmatrix}
   a_{1,1} & a_{1,2} & a_{1,3} & \ldots & a_{1,n}  \\
      0 & a_{2,2} & a_{2,3} & \ldots & a_{2,n}   \\
		  \vdots & \ddots & \ddots &  & \vdots  \\
		  \vdots &  & \ddots & a_{n-1,n-1} & a_{n-1,n}  \\
		  0 & \ldots & \ldots & 0 &  0 \\
   \end{pmatrix}$, with $a_{i,j} \in \mathbb{K}$ and $a_{i,i} \neq 0$. Since $\mathbb{K}$ is finite, there exists an irreducible monic polynomial $P \in \mathbb{K}[X]$ of degree $n$ (see for instance \cite{G} exercise VII.9). By the previous proposition, there exists $(\alpha_{1},\ldots,\alpha_{n}) \in \mathbb{K}^{n}$ such that the characteristic polynomial of the matrix $D=\begin{pmatrix}
   \alpha_{1} & \alpha_{2} & \ldots & \ldots & \alpha_{n}  \\
       a_{1,1} & a_{1,2} & \ldots & \ldots & a_{1,n}  \\
		   0 & a_{2,2} & \ldots & \ldots & a_{2,n}  \\
			 \vdots & \ddots & \ddots &  &  \vdots \\
		  0 & \ldots & 0 & a_{n-1,n-1} & a_{n-1,n}  \\
		   \end{pmatrix}$ is equal to $P$. In paticular, $\pi_{D}=\chi_{D}$, since $\pi_{D}$ divides $\chi_{D}$ (by Cayley-Hamilton theorem) and $\chi_{D}$ irreducible. We have: $B={}^t\!C(X^{n}) \times D$.

\end{proof}

\begin{cor}

Let $n \geq 3$ and $\mathbb{K}$ a finite field. Let $A$ be a triangularizable square matrix of size $n$ satisfying ${\rm rg}(A)=n-2$ and $\chi_{A}(X)=X^{2}R(X)$, with $R(0)\neq 0$. The matrix $A$ can be written as a product of two matrices similar to companion matrices. 

\end{cor}

\begin{proof}

The matrix $A$ is similar to $B:=\begin{pmatrix}
   a_{1,1} & \ldots & \ldots & a_{1,n-2} & a_{1,n-1} & a_{1,n}  \\
      0 & a_{2,2} & \ldots & a_{2,n-2} & a_{2,n-1} & a_{2,n}   \\
		  \vdots & \ddots & \ddots & &  & \vdots  \\
		  \vdots &  & \ddots & a_{n-2,n-2} & a_{n-2,n-1} & a_{n-2,n}  \\
		  0 & \ldots & \ldots & 0 &  0 & 0 \\
			0 & \ldots & \ldots & 0 &  0 & 0 \\
   \end{pmatrix}$, with $a_{i,j} \in \mathbb{K}$ and $a_{i,i} \neq 0$. Since $\mathbb{K}$ is finite, there exists an irreducible monic polynomial $P \in \mathbb{K}[X]$ of degree $n-1$ (see \cite{G} exercise VII.9). By the previous proposition, there exists $(\alpha_{1},\ldots,\alpha_{n-1}) \in \mathbb{K}^{n-1}$ such that the characteristic polynomial of the matrix $D=\begin{pmatrix}
   \alpha_{1} & \alpha_{2} & \ldots & \ldots & \alpha_{n-1}  \\
       a_{1,1} & a_{1,2} & \ldots & \ldots & a_{1,n-1}  \\
		   0 & a_{2,2} & \ldots & \ldots & a_{2,n-1}  \\
			 \vdots & \ddots & \ddots &  &  \vdots \\
		  0 & \ldots & 0 & a_{n-2,n-2} & a_{n-2,n-1}  \\
		   \end{pmatrix}$ is equal to $P$. In paticular, $\pi_{D}=\chi_{D}$, since $\pi_{D}$ divides $\chi_{D}$ (by Cayley-Hamilton theorem) and $\chi_{D}$ irreducible. We have: 
\[B={}^t\!C(X^{n}) \times\underbrace{\left (
   \begin{array}{c|c}
      D & x \\
      \hline
      0_{n-1,1} & 0\\ 
   \end{array}
\right)}_{E},\]
\noindent with $x={}^t\!\begin{pmatrix}
   0 & a_{1,n} & \ldots & a_{n-2,n} 
		   \end{pmatrix}$. We have $\chi_{E}(X)=XP(X)$. By Cayley-Hamilton theorem, the polynomial $\pi_{E}(X)$ divides $\chi_{E}(X)$. Besides, $\pi_{E}(X)$ is a multiple of the lowest common multiple of $X$ and $\pi_{D}(X)=P(X)$ (since $E$ is block-triangular). Moreover, $0 \notin {\rm Sp}(D)$ (otherwise $0$ would be a root of $P$ and $P$ would be reducible, since $n-1 \geq 2$). So, this lowest common multiple is $XP(X)=\chi_{E}(X)$. Hence, $\pi_{E}$ is equal to $\chi_{E}$ and $A$ can be written as a product of two matrices similar to companion matrices.
\end{proof}

\end{document}